\documentclass[12pt,a4paper,reqno,final]{amsart}

\usepackage[margin=2.5cm,headsep=1cm]{geometry}

\usepackage{amsmath,amssymb,amsthm, mathtools}

\usepackage{enumitem}
\setlist[enumerate]{font=\normalfont}

\usepackage[english]{babel}

\usepackage[draft=false]{hyperref}                                                
\usepackage{bookmark}
\usepackage[inline]{showlabels}
\showlabels{cite}
\usepackage[dvipsnames]{xcolor}            
\hypersetup{                                                         
  pdfauthor={Kevin Aguyar Brix},
  pdftitle={Sturmian subshifts and their C*-algebras},
    colorlinks,
    linkcolor={blue!50!black},
    citecolor={red!50!black},
    urlcolor={blue!80!black},
    linktocpage                                                      
}
\usepackage[capitalise, noabbrev, nameinlink]{cleveref}                           

\usepackage{verbatim}                                                

\newcommand{\N}{\mathbb{N}}
\newcommand{\Z}{\mathbb{Z}}
\newcommand{\Q}{\mathbb{Q}}
\newcommand{\K}{\mathbb{K}}
\newcommand{\T}{\mathbb{T}} 
 
\newcommand{\I}{\mathcal{I}}
\newcommand{\G}{\mathcal{G}} 
\newcommand{\D}{\mathcal{D}} 
\newcommand{\OO}{\mathcal{O}} 
 
\newcommand{\X}{\mathsf{X}} 
\newcommand{\LL}{\mathsf{L}} 

\newcommand{\0}{\mathtt{0}}
\newcommand{\1}{\mathtt{1}}

\DeclareMathOperator{\id}{id}
\DeclareMathOperator{\orb}{orb}

\let\tilde\widetilde
\let\leq\leqslant{}     
\let\geq\geqslant{}
\let\subset\subseteq{}      
{}

\newcommand{\widesim}[2][1.5]{\mathrel{\overset{#2}{\scalebox{#1}[1]{$\sim$}}}}

\usepackage{tikz}	
\usetikzlibrary{cd, matrix,arrows,decorations.markings}
\usepgflibrary{arrows.meta}

\newtheorem{lemma}{Lemma}
\newtheorem{corollary}[lemma]{Corollary}
\newtheorem{theorem}[lemma]{Theorem}
\newtheorem{proposition}[lemma]{Proposition}

\theoremstyle{definition}
\newtheorem{definition}[lemma]{Definition}

\newtheorem{example}[lemma]{Example}
\newtheorem{construction}[lemma]{Construction}
\newtheorem{remark}[lemma]{Remark}

\numberwithin{equation}{section}
\numberwithin{lemma}{section}

\makeatletter
\renewcommand{\tocsection}[3]{
  \indentlabel{\@ifnotempty{#2}{\ignorespaces#1 #2.\quad}}#3\dotfill%
}
\renewcommand{\tocsubsection}[3]{%
  \indentlabel{\@ifnotempty{#2}{\ignorespaces#1 #2.\quad}}#3\dotfill%
}
\makeatother

\makeatletter
\let\origsection\section
\renewcommand\section{\@ifstar{\starsection}{\nostarsection}}
\newcommand\sectionspace{\vspace{0.5ex}}
\newcommand\nostarsection[1]{\sectionspace\origsection{#1}\sectionspace}
\newcommand\starsection[1]{\sectionspace\origsection*{#1}\sectionspace}
\makeatother

\makeatletter\g@addto@macro\bfseries{\boldmath}\makeatother

\makeatletter
\@namedef{subjclassname@2020}{%
  \textup{2020} Mathematics Subject Classification}
\makeatother

\makeatletter
\def\subsection{\@startsection{subsection}{2}{\z@}{-3.25ex plus -1ex
  minus -.2ex}{1.5ex plus .2ex}{\normalsize\bf}}
\makeatother

\begin{document}

\begin{abstract}
  This paper investigates the structure of $C^*$-algebras built from one-sided Sturmian subshifts.
  They are parametrised by irrationals in the unit interval and built from a local homeomorphism associated to the subshift.
  We provide an explicit construction and description of this local homeomorphism. 
  The $C^*$-algebras are $^*$-isomorphic exactly when the systems are conjugate,
  and they are Morita equivalent exactly when the defining irrationals are equivalent (this happens precisely when the systems are flow equivalent). 
  Using only elementary dynamical tools, we compute the dynamic asymptotic dimension of the (groupoid of the) local homeomorphism to be one,
  and by a result of Guentner, Willett, and Yu, it follows that the nuclear dimension of the $C^*$-algebras is one.
\end{abstract}


\title{Sturmian subshifts and their $C^*$-algebras}

\author[K.A.~Brix]{Kevin Aguyar Brix}
\address{School of Mathematics and Statistics, University of Glasgow, University Place, G12 8QQ, United Kingdom}
\email{kabrix.math@fastmail.com}
\dedicatory{Para Clelia, en admiraci\'on y respeto. San Genaro siempre estuvo cerca.}

\maketitle

\setcounter{tocdepth}{1}


\section{Introduction}

Sturmian subshifts were introduced by Morse and Hedlund~\cite{Morse-Hedlund} as the first class of symbolic dynamical systems to be systematically studied.
They are canonical examples of Cantor minimal systems~\cite{GPS1995, Skau2000}, they have vanishing topological entropy, 
and they occur naturally as symbolic representations of the irrational rotation on the circle~\cite{Fogg,Berstel-Seebold,LM}.
They also appear as Denjoy homeomorphisms on the circle restricted to their unique minimal invariant Cantor sets, see~\cite{Putnam-Schmidt-Skau} for an illuminating account.
Symbolic dynamics befriended $C^*$-algebra theory with the introduction of Cuntz--Krieger algebras from (irreducible) shifts of finite type~\cite{CK80} in 1980,
and this provided the first large class of simple and purely infinite $C^*$-algebras.
Matsumoto and Carlsen are among the first to associate $C^*$-algebras to general subshifts in the same spirit, see e.g.~\cite{Mat2002,Carlsen2008,Carlsen-Matsumoto2004}. 
However, much of Matsumoto's work is concerned with systems satisfying a certain Condition (I) which is not enjoyed by Sturmian subshifts.

Nuclear dimension for $C^*$-algebras was introduced by Winter and Zacharias~\cite{Winter-Zacharias} as a noncommutative analogue of topological covering dimension,
and it has become a salient concept in the classification programme for nuclear $C^*$-algebras:
simple, separable, unital $C^*$-algebras satisfying the UCT are classified by $K$-theory and traces
exactly when they have \emph{finite} nuclear dimension~\cite{Kirchberg, Phillips, Gong-Lin-Niu, Elliott-Gong-Lin-Niu, TWW}.
The work going into determining the precise values culminated in the discovery that zero and one are the only possible finite values~\cite{Matui-Sato, Sato-White-Winter, BBSTWW, CETWW, Castillejos-Evington}
(see also the expository~\cite{WinterICM}),
coupled with the fact that only AF-algebras have nuclear dimension zero.
These remarkable results naturally lead the attention to the settting of nonsimple $C^*$-algebras with some recent advances~\cite{Brake-Winter, Easo, Evington, Bosa-Gabe-Sims-White, Ruiz-Sims-Tomforde}.
For a commutative $C^*$-algebra, the nuclear dimension coincides with the covering dimension of its spectrum
but in general not much is known about the precise value for nonsimple $C^*$-algebras.
The relationship between finite nuclear dimension and classification of nonsimple $C^*$-algebras remains an open problem.

This article adresses the construction and structure of $C^*$-algebras associated to one-sided Sturmian subshifts.
They are not crossed products.
Instead they are built from the topological groupoid of a certain local homeomorphism, a dynamical system abstractly associated to any one-sided subshift as in~\cite{Brix-Carlsen}.
This dynamical system, called a cover, is in general difficult to describe concretely,
though for sofic systems it coincides with the left Krieger graph (a shift of finite type).
Here we provide the first concrete description beyond the sofic case:
for one-sided Sturmian subshifts the cover is a union of the \emph{two-sided} Sturmian subshift together with a countable discrete orbit. 

The $C^*$-algebras we consider are parametrised over irrationals in the unit interval, 
and they are almost never $^*$-isomorphic (not even Morita equivalent), they have a unique ideal ($^*$-isomorphic to the compact operators on separable Hilbert space),
and the concrete characterisation of the cover allows us to recover part of an interesting but unpublished result of Carlsen~\cite{Carlsen2006} (using different methods):
the $C^*$-algebra is an extension of a simple crossed product by the compact operators.
Moreover, we deduce several structural results. 
The $C^*$-algebras are all infinite (though not properly infinite) and have real rank zero and stable rank two. 
Finally, using only elementary dynamical tools we compute the dynamic asymptotic dimension of the associated groupoid;
using a result of Guentner, Willett, and Yu~\cite{Guentner-Willett-Yu} we then infer that the nuclear dimension of the $C^*$-algebras is one.
The precise dimension values do not seem to be derivable from any of the general theorems available in the literature at the moment, so the method of proof is interesting in itself.
It utilises the particular structure of the Sturmian subshifts, and it would be interesting to see if a similar strategy can be employed for larger classes of dynamical systems.

The article is structured as follows:
after introducing the Sturmian subshifts in~\cref{sec:prelim},
we characterise when these systems are conjugate and flow equivalent in~\cref{sec:dynamics}.
In~\cref{sec:cover} we carefully construct and describe the cover associated to the one-sided Sturmian subshift.
This allows us to build the groupoid and $C^*$-algebra in~\cref{sec:algebras}
and here we also derive many structural results including the nuclear dimension.

\section*{Acknowledgements}
I would like to thank S\o ren Eilers for directing me towards the Sturmian subshifts and for providing me with initial references,
and I thank Jamie Gabe for many helpful discussions.
I gratefully acknowledge the support of the Carlsberg Foundation through an internationalisation fellowship.

\section{Preliminaries} \label{sec:prelim}

Let $\Z$, and $\N$, and $\N_+$ denote the set of integers, the set of nonnegative and positive integers, respectively.
We identify the circle $\T$ with $\mathbb{R}/ \Z$.

\subsection{Sturmian shift spaces}
We start by defining and providing a brief overview of Sturmian subshifts.
For a thorough treatment the reader is refered to, e.g.~\cite[Chapter 6]{Fogg} or~\cite[Chapter 13.7]{LM}.

Let $\alpha\in (0,1)\setminus \Q$ and consider the rigid rotation $R_\alpha\colon \T \to \T$ given by $R_\alpha(t) = t + \alpha~(\textrm{mod}~1)$, for $t\in \T$.
Consider the partition $\{[0,1 - \alpha), [1 - \alpha, 1) \}$ of the unit interval and define a coding map $I_\alpha\colon \T \to \{\0,\1\}$ by
\begin{equation}
  I_\alpha(t) = 
  \begin{cases}
    \0 & \textrm{if } t\in [0, 1 - \alpha), \\
    \1 & \textrm{if } t\in [1 - \alpha, 1),
  \end{cases}
\end{equation}
for $t\in \T$.
The \emph{one-sided Sturmian subshift} with parameter $\alpha$ is the space 
\begin{equation}
  \X_\alpha \coloneqq \overline{ \big\{ \big\{ I_\alpha(R_\alpha^i(t))\big\}_{i\in \N} : t\in \T \big\} } \subset \{\0,\1\}^\N,
\end{equation}
with the shift operation $\sigma_\alpha\colon \X_\alpha \to \X_\alpha$ given by $\sigma_\alpha(x)_i = x_{i + 1}$, for $x = (x_i)_{i\in \N}\in \X_\alpha$ and for $i\in \N$. 
In particular, if $x = \big\{ I_\alpha(R_\alpha^i(t))\big\}_i$ for some $t\in \T$, then 
\begin{equation}
  \sigma_\alpha(x) = \big\{ I_\alpha(R_\alpha^{i + 1}(t))\big\}_{i\in \N} = \big\{ I_\alpha(R_\alpha^i(t + \alpha))\big\}_{i\in \N},
\end{equation}
where, as before, $t + \alpha$ is to be understood modulo $1$.
The shift operation $\sigma_\alpha$ is an expansive and locally injective surjection.
The \emph{orbit} of $x\in \X_\alpha$ is 
\begin{equation} 
  \orb_\alpha(x) \coloneqq \bigcup_{k,l\in \N} \sigma_\alpha^{-l}(\sigma_\alpha^k(x)),
\end{equation}
and $x'\in \orb_\alpha(x)$ exactly if there are $k,l\in \N$ such that $\sigma_\alpha^l(x') = \sigma_\alpha^k(x)$.

The \emph{two-sided Sturmian subshift} is given as the projective limit $\Lambda_\alpha = \varprojlim (\X_\alpha, \sigma_\alpha)$
or, equivalently (and more concretely), as
\begin{equation}
  \Lambda_\alpha \coloneqq \overline{ \big\{ \big\{ I_\alpha(R_\alpha^i(t)) \big\}_{i\in \Z} : t\in \T \big\} } \subset \{\0,\1\}^\Z,
\end{equation}
with a shift operation $\sigma_\alpha\colon \Lambda_\alpha \to \Lambda_\alpha$ given by $\sigma_\alpha(x)_i = x_{i + 1}$,
for $x = (x_i)_{i\in \Z}\in \Lambda_\alpha$ and $i\in \Z$.
This is a homeomorphism.
We shall use the same symbol $\sigma_\alpha$ for the shifts on the one-sided and the two-sided systems, 
and this should cause no confusion.
There is a surjective continuous map $\rho_\alpha\colon \Lambda_\alpha \to \X_\alpha$ given by $\rho(x) = x_{[0,\infty)}$ for $x\in \Lambda_\alpha$.
It intertwines the shift operations (it is a factor map), and we shall refer to it as the canonical truncation.

A finite sequence $\mu = \mu_0 \dots\mu_n$ of length $|\mu| = n+1$ is an admissible word of $\X_\alpha$ if it appears in some infinite sequence $x\in \X_\alpha$.
Let $\LL(\X_\alpha)$ be the collection of admissible words.
A basis for the topology on $\X_\alpha$ is then given by cylinder sets of the form
\begin{equation} 
  Z(\mu) \coloneqq \{ x\in \X_\alpha : x_{[0, |\mu|)} = \mu \},
\end{equation}
where $\mu\in \LL(\X_\alpha)$.
As a space $\X_\alpha$ is homeomorphic to the Cantor space (in particular, there are no isolated points),
and the system $(\X_\alpha, \sigma_\alpha)$ is minimal in the sense that every orbit is dense.
In fact, if $\mu$ is an admissible word, then there is number $\beta(\mu)\in \N$ such that if $x\in \X_\alpha$, then $\mu$ appears in $x_{[0, \beta(\mu))}$.

The one-sided shift map $\sigma_\alpha$ is one-to-one except at a single point $\omega_\alpha\in \X_\alpha$ corresponding to $t = \alpha$ above.
This is the unique element satisfying 
\begin{equation} 
  \sigma_\alpha^{-1}(\omega_\alpha) = \{ \0\omega_\alpha, \1\omega_\alpha\},
\end{equation}
and we refer to $\omega_\alpha$ as a \emph{branch point}.\footnote{In~\cite[Chapter 6]{Fogg},
this point is called \emph{(left) special} and denoted $\ell$.}
With our choice of coding map $I_\alpha$, the preimage $\0 \omega_\alpha$ corresponds to $t = 0$,
and $\1\omega_\alpha$ is added in the closure;
however, using an alternative coding map $I'_\alpha$ defined in terms of the partition $\{(0, 1 - \alpha], (1 - \alpha, 1] \}$
the preimage $\1\omega_\alpha$ would correspond to $t = 0$, and $\0 \omega_\alpha$ would be added in the closure.

The system $(\X_\alpha, \sigma_\alpha)$ is aperiodic in the sense that it contains no eventually periodic points,
but the branch point $\omega_\alpha$ is isolated in past equivalence so $\X_\alpha$ does not satisfy Matsumoto's condition (I), cf.~\cite[Proposition 2.10]{Brix-Carlsen}.

The \emph{ordered cohomology} of a Sturmian system with parameter $\alpha\in (0,1)\setminus \Q$ is the ordered group $\Z + \alpha \Z$
(the ordering is inherited from the real line) with a distinguished order unit $1\in \Z + \alpha \Z$.
This is a simple dimension group with no infinitesimals, cf.~\cite[Section 1]{GPS1995}.

\subsection{Continued fractions}
The group $\textrm{GL}(2,\Z)$ ($2\times 2$ matrices with integer entries whose determinant is $\pm 1$) acts on irrationals by linear fractional transformations, i.e.
\begin{equation}
  \begin{pmatrix}
    a & b \\
    c & d
  \end{pmatrix} \alpha
  = \frac{a \alpha + b}{c \alpha + d},
\end{equation}
where $\alpha$ is irrational, and $a,b,c,d\in \Z$ satisfy $ad - cb = \pm 1$.
Two irrationals $\alpha$ and $\beta$ are \emph{equivalent} if they are in the same orbit under this action;
we write $\alpha \sim \beta$.

Every irrational number $\alpha$ has a unique continued fraction representation $[a_0; a_1, a_2,\ldots]$, with $\alpha_0\in \Z$ and $a_i\in \N_+$ for all $i\in \N_+$,
and two irrationals are equivalent if and only if their continued fraction representations $[a_0; a_1, a_2, \ldots]$ and $[b_0; b_1, b_2, \ldots]$ are tail equivalent, 
cf.~\cite[Section 10.11]{Hardy-Wright}.

\begin{example}
  The Fibonacci substitution 
  \begin{equation} 
    \tau\colon 
    \begin{cases}
      \texttt{0} \mapsto 01 \\
      \texttt{1} \mapsto 0
    \end{cases}
  \end{equation}
  with fixed point  
  \begin{equation} 
    x = \texttt{0 1 0 0 1 0 1 0 0 1} \dots
  \end{equation}
  is called the Fibonacci sequence.
  It is a Sturmian sequence with parameter $\alpha = (3 - \sqrt{5})/2$,
  and the continued fraction expansion is $\alpha = [0; 2, 1,1,\dots]$.
  In particular, it is ultimately periodic.
  See e.g. the very careful exposition in~\cite[Example 2.1.1]{Berstel-Seebold}.
\end{example}


\section{Dynamical relations of Sturmian systems} \label{sec:dynamics}
In this section, we discuss the dynamical relations between Sturmian systems of different parameters.
Most of the results are well known to experts and the section serves to highlight complete invariants of conjugacy and flow equivalence. 
The final result on rigidity of continuous orbit equivalence is new.

A pair of two-sided Sturmian systems $(\Lambda_\alpha, \sigma_\alpha)$ and $(\Lambda_\beta, \sigma_\beta)$ are \emph{conjugate} 
if there is a homeomorphism $h\colon \Lambda_\alpha \to \Lambda_\beta$ such that $h\circ \sigma_\alpha = \sigma_\beta \circ h$.
Since the shift is a homeomorphism, the reversed action $(\Lambda_\alpha, \sigma_\alpha^{-1})$ defines a new dynamical systems,
and we say that $(\Lambda_\alpha, \sigma_\alpha)$ and $(\Lambda_\beta, \sigma_\beta)$ are \emph{flip conjugate}
if either $(\Lambda_\alpha, \sigma_\alpha)$ or $(\Lambda_\alpha, \sigma_\alpha^{-1})$ is conjugate to $(\Lambda_\beta, \sigma_\beta)$.
Note that $\sigma_\alpha^{-1}$ corresponds to reversing the rotation of the circle, so it is conjugate to $\sigma_{(1-\alpha)}$.
Moreover, the map $\T \to \T$ sending $t$ to $1 - t$ (mod $1$) actually induces a conjugacy between $(\X_\alpha, \sigma_\alpha)$ and $(\X_{(1-\alpha)}, \sigma_{(1-\alpha)})$.
This map just implements an interchanging of the two symbols.
In particular, flip conjugacy implies conjugacy.
The systems are \emph{orbit equivalent} if there is a homeomorphism $h\colon \Lambda_\alpha \to \Lambda_\beta$ 
which maps the orbit of $x\in \Lambda_\alpha$ onto the orbit of $h(x)\in \Lambda_\beta$.

\begin{proposition}\label{thm:two-sided-conjugacy} 
  Let $(\Lambda_\alpha, \sigma_\alpha)$ and $(\Lambda_\beta, \sigma_\beta)$ be two-sided Sturmian subshifts.
  The following are equivalent:
  \begin{enumerate}[label=(\arabic*)]
    \item \label{i:alpha-beta} $\alpha = \beta$ or $\alpha = 1 - \beta$;
    \item \label{i:conjugate} $(\Lambda_\alpha, \sigma_\alpha)$ and $(\Lambda_\beta, \sigma_\beta)$ are two-sided conjugate;
    \item \label{i:OE} $(\Lambda_\alpha, \sigma_\alpha)$ and $(\Lambda_\beta, \sigma_\beta)$ are orbit equivalent; and
    \item \label{i:dimension-groups-unit} the dimension groups $(\Z + \alpha \Z, 1) \cong (\Z + \beta \Z, 1)$ are isomorphic as ordered groups with distinguished order units.
  \end{enumerate}
\end{proposition}

\begin{proof}
  \labelcref{i:alpha-beta} $\implies$ \labelcref{i:conjugate}:
  The map $\T \to \T$ that sends $t$ to $1 - t$ induces a conjugacy between $(\Lambda_\alpha, \sigma_\alpha)$ and $(\Lambda_{(1 - \alpha)}, \sigma_{(1-\alpha)})$
  which simply interchanges the two symbols.

  \labelcref{i:conjugate} $\implies$ \labelcref{i:OE}: 
  This is clear.

  \labelcref{i:OE} $\iff$ \labelcref{i:dimension-groups-unit}:
  The dimension group of a Sturmian subshift contains no infinitesimals, so it follows from ~\cite[Theorems 2.1 and 2.2]{GPS1995} that two systems are orbit equivalent
  exactly when their dimension groups are order isomorphic in a way which respects the distinguished order units.
 
  \labelcref{i:dimension-groups-unit} $\implies$ \labelcref{i:alpha-beta}:
  If there is an order isomorphism $\Z +\alpha \Z \to \Z + \beta \Z$ which maps $1$ to $1$, then there are integers $n,m,p,q\in \Z$ such that $\alpha = n + \beta m$ and $\beta = p + \alpha q$.
  This means that $\alpha = (n + pm) + \alpha qm$, so $qm = 1$.
  Since $\alpha\in (0,1)$, we see that if $m = 1$, then $\alpha = n + \beta = \beta$, and if $m = -1$, then $\alpha = (n + pm) - \beta = 1 - \beta$.
\end{proof}

\begin{remark}
  This characterisation of conjugacy is well known, see e.g.~\cite[p. 26]{Dartnell-Durand-Maass}.
\end{remark}

The \emph{suspension} (or mapping torus) of $(\Lambda_\alpha, \sigma_\alpha)$ is the compact Hausdorff space $\Lambda_\alpha \times \mathbb{R} / \sim$
where $(x,s + 1) \sim (\sigma_\alpha(x), s)$. 
Let $[x,s]$ denote the equivalence class of $(x,s)$.
The suspension carries a flow (an action of $\mathbb{R}$) given by $\phi_{s'}[x,s] = [x, s' + s]$, for $s'\in \mathbb{R}$,
and two systems are \emph{flow equivalent} if there is a homeomorphism between their suspensions which respects the flow orbits in an orientation-preserving way, cf.~\cite[Section 13.6]{LM}.

\begin{proposition} 
  Let $(\Lambda_\alpha, \sigma_\alpha)$ and $(\Lambda_\beta, \sigma_\beta)$ be two-sided Sturmian subshifts.
  The following are equivalent:
  \begin{enumerate}[label=(\arabic*)]
    \item \label{i:equivalent-irrationals} $\alpha \sim \beta$ (equivalence of irrationals);
    \item \label{i:flow} $\Lambda_\alpha$ and $\Lambda_\beta$ are flow equivalent;
    \item \label{i:suspension-homeomorphism} the suspension spaces of $\Lambda_\alpha$ and $\Lambda_\beta$ are homeomorphic; and
    \item \label{i:dimension-groups} the dimension groups $\Z + \alpha \Z \cong \Z + \beta \Z$ are isomorphic as ordered groups (not necessarily preserving the distinguished order unit).
  \end{enumerate}
\end{proposition}

\begin{proof}
  \labelcref{i:equivalent-irrationals} $\iff$ \labelcref{i:dimension-groups}:
  The fact that dimension groups of the form $\Z + \alpha \Z$ and $\Z + \beta\Z$ for irrationals $\alpha$ and $\beta$ are isomorphic as ordered groups 
  exactly when $\alpha$ and $\beta$ are equivalent irrationals is stated in~\cite[Theorem 3.2]{Effros-Shen} (see also~\cite[p. 148]{Douglas}).

  \labelcref{i:flow} $\iff$ \labelcref{i:dimension-groups}:
  By~\cite[Theorem 2.6]{GPS1995}, the systems $(\Lambda_\alpha, \sigma_\alpha)$ and $(\Lambda_\beta, \sigma_\beta)$ are Kakutani strong orbit equivalent 
  (in the sense that $(\Lambda_\alpha, \sigma_\alpha)$ and $(\Lambda_\beta, \sigma_\beta)$ are strong orbit equivalent to systems which are Kakutani equivalent)
  if and only if their dimension groups are isomorphic as ordered groups.
  Strong orbit equivalent Sturmian systems are conjugate (cf.~\cref{thm:two-sided-conjugacy}), and Kakutani equivalence is the same as flow equivalence
  (see e.g.~\cite[Remark 2.4]{Kosek-Ormes-Rudolph} and~\cite{Parry-Sullivan})
  so the equivalence follows.

  \labelcref{i:equivalent-irrationals} $\iff$ \labelcref{i:suspension-homeomorphism}:
  This is known as Fokkink's theorem~\cite[Theorem 4.6]{Barge-Williams}.
  Although it is not mentioned by Barge and Williams in their paper, the Denjoy continuum of a rigid rotation may be identified with the suspension space of a Sturmian subshift, 
  see e.g.~\cite[Section 3.6]{Fuhrmann-Groger-Jager}.
\end{proof}

Let us now consider dynamical relations between one-sided Sturmian systems.
A homeomorphism $h\colon \X_\alpha \to \X_\beta$ is a \emph{continuous orbit equivalence} if there are continuous maps 
$k, l\colon \X_\alpha \to \N$ and $k',l'\colon \X_\beta \to \N$ such that
\begin{align} 
  \sigma_\beta^{l(x)}(h(x)) &= \sigma_\beta^{k(x)}(h(\sigma_\alpha(x))), \\
  \sigma_\alpha^{l'(y)}(h^{-1}(y)) &= \sigma_\alpha^{k'(y)}(h^{-1}(\sigma_\beta(y))),
\end{align}
for $x\in \X_\alpha$ and $y\in \X_\beta$,
in which case $\X_\alpha$ and $\X_\beta$ are continuously orbit equivalent.
We say $h$ is an \emph{eventual conjugacy} if we can choose $l = k + 1$ and $l' = k' + 1$
in which case $\X_\alpha$ and $\X_\beta$ are eventually conjugate.
Finally, $h$ is a conjugacy if $h\circ \sigma_\alpha = \sigma_\beta\circ h$ 
in which case we say that $\X_\alpha$ and $\X_\beta$ are one-sided conjugate.

\begin{proposition} 
  Let $(\X_\alpha, \sigma_\alpha)$ and $(\X_\beta, \sigma_\beta)$ be one-sided Sturmian subshifts.
  \begin{enumerate}[label=(\roman*)]
    \item \label{i:eventual-conjugacy} An eventual conjugacy $h\colon \X_\alpha \to \X_\beta$ is a conjugacy.
    \item \label{i:coe} If $(\X_\alpha, \sigma_\alpha)$ and $(\X_\beta, \sigma_\beta)$ are continuously orbit equivalent, then they are conjugate.
  \end{enumerate}
\end{proposition}

\begin{proof}
  \labelcref{i:eventual-conjugacy}:
  Let $h\colon \X_\alpha \to \X_\beta$ be an eventual conjugacy with continuous cocycles $k\colon \X_\alpha \to \N$ and $k'\colon \X_\beta\to \N$.
  Take $x\in \X_\alpha$ such that $h(x) \notin \orb_\beta(\omega_\beta)$.
  Then $h(\sigma_\alpha(x)) \notin \orb(\omega_\beta)$ and the relation
  \begin{equation}
    \sigma_\beta^{k(x) + 1}(h(x)) = \sigma_\beta^{k(x)}(h(\sigma_\alpha(x)))
  \end{equation}
  implies that $k|_{\orb_\alpha(x)} = 0$ (since $h(x)$ and $h(\sigma_\alpha(x))$ have unique pasts, cf.~\cref{def:unique-past}).
  Since $\orb_\alpha(x)$ is dense in $\X_\alpha$ and $k$ is continuous we have $k = 0$, so $h$ is a conjugacy.

  \labelcref{i:coe}
  This proof uses results and terminology from the next section.
  Suppose $h\colon \X_\alpha \to \X_\beta$ be a continuous orbit equivalence.
  By a lifting result~\cite[Lemma 6.3]{Brix-Carlsen}, there is an induced continuous orbit equivalence $\tilde{h}\colon \tilde{\X}_\alpha \to \tilde{\X}_\beta$ between the covers
  satifying $\tilde{h}\circ \pi_\alpha = \pi_\alpha\circ h$.
  Since $\tilde{h}$ maps isolated points to isolated point, the map $\tilde{h}$ restricts to a continuous orbit equivalence $\tilde{h}\colon \tilde{\Lambda}_\alpha \to \tilde{\Lambda}_\beta$.
  Via the conjugacies $\phi_\alpha\colon \Lambda_\alpha \to \tilde{\Lambda}_\alpha$ and $\phi_\beta \colon \Lambda_\beta \to \tilde{\Lambda}_\beta$, 
  we therefore obtain an orbit equivalence $\tilde{h}\colon \Lambda_\alpha \to \Lambda_\beta$ with a continuous cocycle $c\colon \Lambda_\alpha\to \Z$ such that 
  \begin{equation}
    \sigma_\beta^{c(x)} \tilde{h}(x) = \tilde{h}(\sigma_\alpha(x)),
  \end{equation}
  for $x\in \Lambda_\alpha$.
  Therefore, the systems $(\Lambda_\alpha, \sigma_\alpha)$ and $(\Lambda_\beta, \sigma_\beta)$ are two-sided continuously orbit equivalent and hence flip conjugate, by~\cite[Theorem 2.4]{GPS1995}.
  This means that $\alpha = \beta$ or $\alpha = 1 - \beta$, so $(\X_\alpha, \sigma_\alpha)$ and $(\X_\beta, \sigma_\beta)$ are conjugate.
\end{proof}

\section{The cover of a Sturmian subshift}\label{sec:cover}
In this section, we first review the construction of the cover of a general subshift from~\cite[Section 2.1]{Brix-Carlsen}
and then analyse the cover associated to a Sturmian system in detail.

Fix $N\in \N_+$ and let $\sigma$ be the shift operation on $\{1,\dots,N\}^\N$, i.e. $\sigma(x)_i = x_{i+1}$ for any $x = (x_i)_{i\in \N} \in \{1,\dots,N\}^\N$.
To any closed subset $\X \subset \{1,\dots,N\}$ which is shift invariant (in the sense that $\sigma(\X) = \X$)
we can associate a totally disconnected compact Hausdorff space $\tilde{\X}$ with a local homeomorphism $\sigma_{\tilde{\X}}$ which we shall also refer to as a shift operation.
There is a surjective continuous map $\pi_\X\colon \tilde{\X} \to \X$ which intertwines the shift operations,
and the pair $(\tilde{\X}, \sigma_{\tilde{\X}})$ is called the \emph{cover} of the subshift $(\X, \sigma|_\X)$.
There is also a (not necessarily continuous) injective function $\iota_\X\colon \X\to \tilde{\X}$ such that $\pi_\X\circ \iota_\X = \id_\X$ and $\iota_\X(\X)$ is dense in $\tilde{\X}$.

\subsection{The cover}
Let us first construct the cover of a one-sided subshift $(\X, \sigma_\X)$ with $\sigma_\X(\X) = \X$.
For $l\in \N$, the $l$-\emph{past} of a point $x\in \X$ is the set 
\[
  P_l(x) \coloneqq \{\mu\in \LL(\X) : |\mu| = l,~\mu x\},
\]
and $x$ is \emph{isolated in past equivalence} if $P_l(x) = P_l(x')$ implies $x = x'$, for some $l\in \N$.
In this case, we say $x$ is isolated in $l$-past equivalence.
Note that if $x$ is isolated in $l$-past equivalence, then it is also isolated in $(l+1)$-past equivalence.

Let $\preceq$ be the partial order on the set $\I = \{ (k,l)\in \N\times \N : k\leq l\}$ given by
\begin{equation} 
  (k_1,l_1) \preceq (k_2,l_2) \iff k_1 \leq k_2~\textrm{and}~l_1-k_1 \leq l_2-k_2,
\end{equation}
for $(k_i,l_i)\in \I$, $i = 1,2$.
Each $(k,l)\in \I$ defines an equivalence relation $\widesim{k,l}$ on $\X$ given by
\begin{equation} 
  x \widesim{k,l} x' \iff x_{[0,k)} = x'_{[0,k)}~\textrm{and}~P_l(\sigma_\X^k(x)) = P_l(\sigma_\X^k(x')).
\end{equation}
Let ${}_k[x]_l$ denote the $\widesim{k,l}$-equivalence class of $x\in \X$, and let ${}_k\X_l = \X / \widesim{k,l}$ be the finite set of $\widesim{k,l}$-equivalence classes.
If $(k_1,l_1) \preceq (k_2, l_2)$, then 
\begin{equation} 
  x \widesim{k_2,l_2} x' \implies x \widesim{k_1,l_1} x',
\end{equation}
so there is a surjection ${}_{(k_1,l_1)}Q_{(k_2,l_2)}\colon {}_{k_2}\X_{l_2} \to {}_{k_1}\X_{l_1}$ given by ${}_{(k_1,l_1)}Q_{(k_2,l_2)}({}_{k_2}[x]_{l_2}) = {}_{k_1}[x]_{l_2}$, for $x\in \X$.
The sets ${}_k\X_l$ together with the maps $Q$ constitute a projective system and we define $\tilde{\X}$ to be the projective limit.

\begin{definition}
  Let $(\X, \sigma_\X)$ be a one-sided subshift with $\sigma_\X(\X) = \X$.
  The cover $\tilde{\X}$ is the projective limit $\varprojlim_{(k,l)\in \I} ({}_k\X_l, Q)$.
\end{definition}
The cover is a second-countable and totally disconnected compact Hausdorff space.
An element of $\tilde{\X}$ is of the form $\tilde{x} \in \prod_{(k,l)\in \I} {}_k \X_l$ with ${}_k \tilde{x}_l = {}_k [x]_l$
for some $x\in \X$ (which we refer to as a \emph{$(k,l)$-representative} of $\tilde{x}$),
and such that if $x$ is a $(k_2,l_2)$-representative of $\tilde{x}$ and $(k_1,l_1)\preceq (k_2,l_2)$, then $x$ is also a $(k_1,l_1)$-representative.

A basis for the topology of $\tilde{\X}$ is given by compact open sets of the form
\begin{equation} 
  U(z, k,l) \coloneqq \{ \tilde{x}\in \tilde{\X} : z \widesim{k,l} {}_k\tilde{x}_l \}
\end{equation}
for $z\in \Z$ and $(k,l)\in \I$.

For each $(k,l)\in \I$ with $k\geq 1$, we observe that
\begin{equation} 
  x \widesim{k,l} x' \implies \sigma_\X(x) \widesim{k-1,l} \sigma_\X(x'),
\end{equation}
so there is a surjective map ${}_k\sigma_l\colon {}_k\X_l \to {}_{k-1}\X_l$ given by ${}_k\sigma_l({}_k[x]_l) = {}_{k-1}[\sigma_\X(x)]_l$ for $x\in \X$.
This defines a \emph{shift operation} $\sigma_{\tilde{\X}}$ on $\tilde{\X}$ given by
\begin{equation} 
  {}_k\sigma_{\tilde{\X}}(\tilde{x})_l = {}_{k+1}\sigma_l({}_{k+1}[{}_{k+1}\tilde{x}_l]_l) = {}_k[\sigma_\X({}_{k+1}\tilde{x}_l)]_l,
\end{equation}
for $\tilde{x}\in \tilde{\X}$, and $\sigma_{\tilde{\X}}$ is a surjective local homeomorphism.

There is a canonical map $\pi_\X \colon \tilde{\X} \to \X$ given as follows:
if $\tilde{x}\in \tilde{\X}$, then $x = \pi_\X(\tilde{x})$ is the unique point satisfying $x_{[0,k)} = ({}_k\tilde{x}_l)_{[0,k)}$, for every $(k,l)\in \I$.
This is a factor map in the sense that it is surjective and satisfies $\pi_\X \circ \sigma_\X = \sigma_{\tilde{\X}}\circ \pi_\X$.
It is injective (and hence a conjugacy) exactly if $(\X, \sigma_\X)$ is a shift of finite type.

Finally, there is an injective function $\iota_\X \colon \X \to \tilde{\X}$ defined as follows:
if $x\in \X$ then $\iota_\X(x)$ is the unique element whose $(k,l)$-representative is $x$ for all $(k,l)\in \I$.
This is a section in the sense that $\pi_\X\circ \iota_\X = \id_\X$ but it is not necessarily continuous.

Before specialising to Sturmian subshifts, we record two lemmas which are valid for general one-sided subshifts.

\begin{lemma}\label{lem:cover-isolated}
  Let $\X$ be a one-sided subshift.
  Any isolated point in the cover $\tilde{\X}$ is contained in the image of $\iota_\X\colon \X\to \tilde{\X}$,
  and each fibre $\pi_\X^{-1}(x)$ contains at most one isolated point.
\end{lemma}

\begin{proof}
  Assume $\tilde{x}\in \tilde{\X}$ is an isolated point.
  This means that
  \begin{equation}
    \{\tilde{x}\} = U(z, k, l),
  \end{equation}
  for some $z\in \X$ and $(k,l)\in \I$, so that ${}_k \tilde{x}_l \widesim{k,l} z$.
  Note that if $(k,l) \preceq (k',l')$, then $U(z, k', l') \subset U(z, k, l)$.
  In particular, ${}_k \tilde{x}_l \widesim{k', l'} z$, for all $(k,l) \preceq (k', l')$.
  So $\tilde{x} = \iota_\X(z)$.

  For each $x\in \X$, we have $\pi_\X^{-1}(x)\cap \iota_\X(\X) = \{\iota_\X(x)\}$, so each fibre contains at most one isolated point.
\end{proof}

Let us say that $\sigma_\X^k$ is \emph{injective at $x\in \X$} if $\sigma_\X^{-k}(x) \coloneqq (\sigma_\X^k)^{-1}(x)$ is a singleton.
The next definition is particularly relevant to Sturmian systems.

\begin{definition} \label{def:unique-past}
  A point $x\in \X$ has a \emph{unique past} if $\sigma_\X^k$ is injective at $x$ for all $k\in \N$.
\end{definition}

A similar definition applies to elements in the cover.

\begin{lemma}\label{lem:unique-past}
  Let $\X$ be a one-sided subshift with $\sigma_\X(\X) = \X$.
  If $\sigma_\X^k$ is injective at $x\in \X$, then $\sigma_{\tilde{\X}}^k$ is injective at any $\tilde{x}\in \pi_\X^{-1}(x)$.
  In particular, if $x\in \X$ has a unique past, then any $\tilde{x}\in \pi_\X^{-1}(x)$ has a unique past.
\end{lemma}

\begin{proof}
  Suppose $\sigma_\X^k$ is injective at $x\in \X$ and let $\mu\in \LL_k(\X)$ be the unique prefix of $x$ with $|\mu| = k$.
  Fix $\tilde{x}\in \pi_\X^{-1}(x)$.
  Let $\tilde{y}, \tilde{z}\in \tilde{\sigma}_\X^{-k}(\tilde{x})$ and observe that
  \begin{equation}
    \pi_\X(\tilde{y}) = \mu x = \pi_\X(\tilde{z}).
  \end{equation}
  By~\cite[Lemma 2.8(i)]{Brix-Carlsen}, we have $\tilde{y} = \tilde{z}$.
  If $x$ has unique past, then this argument applies for all $k\in \N_+$, so we conclude that $\tilde{x}$ has a unique past.
\end{proof}

\subsection{Sturmian systems}
Let us now restrict attention to Sturmian systems.
When the parameter is $\alpha\in (0,1)\setminus \Q$, we denote the cover by $(\tilde{\X}_\alpha, \tilde{\sigma}_\alpha)$, 
the factor map by $\pi_\alpha$, and the injection by $\iota_\alpha$.
Recall that the shift operation $\sigma_\alpha$ is one-to-one everywhere except at the unique branch point $\omega_\alpha\in \X_\alpha$ at which $\sigma_\alpha$ is two-to-one,
cf.~\cite[Theorem 6.1.20]{Fogg}.
We first recall a lemma which will be useful throughout the section (see~\cite[Lemma 2.8]{Brix-Carlsen} for a more general statement and a proof).

\begin{lemma}\label{lem:BC-lemma} 
  Let $(\X_\alpha, \sigma_\alpha)$ be a Sturmian subshift and let $(\tilde{\X}_\alpha, \tilde{\sigma}_\alpha)$ be its cover.
  If $\tilde{x}, \tilde{z}\in \tilde{\X}$ satisfy $\pi_\alpha(\tilde{x}) = \pi_\alpha(\tilde{z})$ and $\tilde{\sigma}_\alpha^k(\tilde{x}) = \tilde{\sigma}_\alpha^l(\tilde{z})$ for some $k,l\in \N$,
  then $\tilde{x} = \tilde{z}$.
\end{lemma}

We first discuss isolated points.

\begin{lemma} \label{lem:Sturmian-cover-isolated}
  For $x\in \X_\alpha$, the fibre $\pi_\alpha^{-1}(x) \subset \tilde{\X}_\alpha$ contains an isolated point if and only if $x\in \orb_\alpha(\omega_\alpha)\subset \X_\alpha$.
\end{lemma}

\begin{proof}
  Let $\omega = \omega_\alpha$ be the branch point.
  If $x = \sigma_\alpha^k(\omega)$ for some $k\in \N$, then $\sigma_\alpha^k(\omega)$ is the unique element satisfying
  \begin{equation}
    P_{k+1}(\sigma_\alpha^k(\omega)) = \{\0 \omega_{[0,k)}, \1\omega_{[0,k)}\},
  \end{equation}
  so $\sigma_\alpha^k(\omega)$ is isolated in $(k + 1)$-past equivalence.
  This implies that 
  \begin{equation} 
    \{ \iota_\alpha(x) \} = U(x, 0, k + 1),
  \end{equation}
  and $\iota_\alpha(x)\in \tilde{\X}_\alpha$ is an isolated point.

  Suppose instead $x\in \sigma_\alpha^{-k}(\omega)$ for some $k\in \N_+$ so that $x = \mu \omega$ with $|\mu| = k$.
  If $x \widesim{k,k} z$ for some $z\in \X_\alpha$, then $z\in Z_\alpha(\mu)$ and 
  \begin{equation} 
    P_k(\sigma_\alpha^k(z)) = P_k(\omega).
  \end{equation}
  Since $\omega$ is isolated in $1$-past equivalence, it is isolated in $k$-past equivalence, so $\sigma_\alpha^k(z) = \omega$.
  Therefore $z = \mu \omega = x$, and 
  \begin{equation} 
    \{ \iota_\alpha(x) \} = U(\mu \omega, k, k)
  \end{equation}
  which means that $\iota_\alpha(x)\in \tilde{\X}_\alpha$ is an isolated point.

  For the converse implication, suppose $x\notin \orb_\alpha(\omega)$.
  By \cref{lem:cover-isolated} we only need to verify that $\iota_\alpha(x)$ is not isolated,
  so take any basic open subset $U(z, k, l)$ containing $\iota_\alpha(x)$, for some $z\in \X_\alpha$ and $(k,l)\in \I$.
  Let $\mu$ be the unique word satisfying $\mu x\in \X_\alpha$ and $|\mu| = l - k$, and put $\nu \coloneqq x_{[0,r)}$.
  Since $\X_\alpha$ is recurrent, there exists $K\in \N$ such that $\sigma_\alpha^K(x) = \mu \nu x'$ for some $x'\in \X_\alpha$.
  Note that $x \neq \nu x'$ because $x$ is aperiodic.
  Now $\sigma_\alpha^{K + |\mu|}(x) = \nu x' \widesim{r,s} x$, so $\iota_\alpha(\nu x')\in U(z, r, s)$,
  and this implies that $\iota_\alpha(x)$ is not isolated.
\end{proof}

\begin{remark}
Let $\orb_\alpha^+(\omega_\alpha) \coloneqq \{ \sigma_\alpha^k(\omega_\alpha) : k\in \N \}$.
We have seen that $x\in \orb_\alpha^+(\omega_\alpha)$ exactly if $x$ is isolated in past equivalence (viz. $\sigma_\alpha^k(\omega_\alpha)$ is isolated in $(k + 1)$-past equivalence),
and $x\notin \orb_\alpha^+(\omega_\alpha)$ exactly if $x$ has a unique past.
\end{remark}

\begin{lemma} 
  The collection $\iota_\alpha(\orb_\alpha(\omega_\alpha))$ of isolated points is dense in $\tilde{\X}_\alpha$.
\end{lemma}

\begin{proof}
  Take a basic open set $U(x,k,l)$ for some $x\in \X_\alpha$ and integers $0\leq k\leq l$.
  We shall find a point $z\in \orb_\alpha(\omega_\alpha)$ such that $x \widesim{k,l} z$;
  this will imply that $i_\alpha(z) \in \iota_\alpha(\orb_\alpha(\omega_\alpha)) \cap U(x,k,l)$.
  We may assume that $x\notin \orb_\alpha(\omega_\alpha)$ so, in particular, $P_l(\sigma_\alpha^k(x)) = \{ \mu\}$ for some word $\mu$.
  As $\orb_\alpha(\omega_\alpha)$ is dense in $\X_\alpha$, there is a point $z'\in \orb_\alpha(\omega_\alpha) \cap Z(\mu)$.
  Since $\X_\alpha$ is recurrent we may assume that $z' = \sigma_\alpha^K(\omega_\alpha)$ for some $K\geq 0$ so that $P_l(\sigma_\alpha^l(z')) = \{\mu\}$.
  Set $z \coloneqq \sigma_\alpha^{l-k}(z')$.
  Then $z_{[0,k)} = x_{[0,k)}$ and $P_l(\sigma_\alpha^k(z)) = P_l(\sigma_\alpha^l(z')) = \{\mu\}$,
  so $x \widesim{k,l} z$ as wanted.
\end{proof}

The next lemma is a well known result which shows that Sturmian subshifts satisfy Property $(\ast)$ of~\cite[Definition 3.1]{Carlsen-Eilers2004}.

\begin{lemma}[Property $(\ast)$] \label{lem:Property(ast)}
  For any word $\mu\in \LL_\alpha$ there is a point $x\in \X_\alpha$ such that $P_{|\mu|}(x) = \{ \mu \}$.
\end{lemma}

\begin{proof}
  Take any $z\notin \orb_\alpha(\omega_\alpha)$.
  As $\X_\alpha$ is minimal, there is $z'\in \orb_\alpha(z)$ such that $z'\in Z_\alpha(\mu)$.
  Then $x = \sigma_\alpha^{|\mu|}(z')\notin \orb_\alpha(\omega_\alpha)$ so $P_{|\mu|}(x) = \{\mu\}$. 
\end{proof}

We can use Property $(\ast)$ to understand the precise structure of the cover.
Loosely speaking, for every $x\in \X_\alpha$ each choice of past generates a distinct element in the cover.
This means that only points in the orbit $\orb_\alpha(\omega_\alpha)$ will have more than a single element in the corresponding fibre
because points outside of this orbit have unique pasts.

\begin{construction}\label{construction}
  Consider the branch point $\omega = \omega_\alpha$ which has exactly two distinct pasts, one corresponding to the prefix \texttt{0} and another to the prefix \texttt{1}.
  We shall explicitly construct an element $\tilde{x}\in \pi_\alpha^{-1}(\omega)$ corresponding to the prefix \texttt{0}; 
  the construction of an element $\tilde{x}'\in \pi_\alpha^{-1}(\omega)$ for the prefix \texttt{1} is analogous.

  We start by defining the representatives, so let $(k,l)\in \I$, i.e. $0 \leq k \leq l$ are integers.
  Take the unique word $\mu \texttt{0}\in \LL_\alpha$ satisfying $\mu \texttt{0} \omega\in \X_\alpha$ and $|\mu \texttt{0}| = l - k$,
  and choose $x_{(k,l)}\in \X_\alpha$ such that $P_l(x_{(k,l)}) = \{ \mu \texttt{0} \omega_{[0, k)}\}$, in accordance with Property $(\ast)$ (\cref{lem:Property(ast)}).
  Set
  \begin{equation} \label{eq:r,s-representatives}
    {}_k \tilde{x}_l \coloneqq \omega_{[0, k)} x_{(k,l)},
  \end{equation}
  and note that ${}_k\tilde{x}_l \not\widesim{k,l} \omega$.

  The notation indicates that each ${}_k\tilde{x}_l$ are representatives of a well defined element $\tilde{x}\in \pi_\alpha^{-1}(\omega)$, and this is what we will now verify.
  Given integers $0\leq k_1\leq l_1$ and $0\leq k_2\leq l_2$ satisfying $(k_1, l_1) \preceq (k_2, l_2)$ (i.e. $k_1 \leq k_2$ and $l_1 - k_1 \leq l_2 - k_2$),
  we must show that 
  \begin{equation}
    {}_{k_1}\tilde{x}_{l_1} \widesim{k_1, l_1} {}_{k_2}\tilde{x}_{l_2}.
  \end{equation}
  Choose the word $\mu \texttt{0}\in \LL_\alpha$ such that $\mu \texttt{0} \omega\in \X_\alpha$ and $|\mu \texttt{0}| = l_1 - k_1$.
  It is clear from the construction that 
  \begin{equation}
    ({}_{k_1}\tilde{x}_{l_1})_{[0, k_1)} = \omega_{[0, k_1)} = ({}_{k_2}\tilde{x}_{l_2})_{[0, k_1)},
  \end{equation}
  so it remains to verify that
  \begin{equation}
    P_{l_1}(\sigma_\alpha^{k_1}({}_{k_1}\tilde{x}_{l_1})) = P_{l_1}(\sigma_\alpha^{k_1}({}_{k_2}\tilde{x}_{l_2})).
  \end{equation}
  But this follows from the observations that
  \begin{equation}
    \sigma_\alpha^{k_1}({}_{k_1}\tilde{x}_{l_1}) = x_{(k_1, l_1)}, \qquad
    \sigma_\alpha^{k_1}({}_{k_2}\tilde{x}_{l_2}) = \omega_{[k_1, k_2)} x_{(k_2, l_2)},
  \end{equation}
  so 
  \begin{equation}
     P_{l_1}(\sigma_\alpha^{k_1}({}_{k_1}\tilde{x}_{l_1})) 
     = \{(\mu \texttt{0} \omega_{[0, k_1)})_{[k_1, k_1 + l_1)}\}
     = P_{l_1}(\sigma_\alpha^{k_1}({}_{k_2}\tilde{x}_{l_2})).
  \end{equation}
  Therefore, we have a well defined element $\tilde{x}\in \tilde{\X}_\alpha$ satisfying $\pi_\alpha(\tilde{x}) = \omega$.
  It is clear from the construction that $\tilde{x}$ is independent of the choice of $x_{(r,s)}$, and that $\tilde{x} \neq \iota_\alpha(\omega)$ so $\tilde{x}$ is not isolated.
\end{construction}

We are now ready to describe the precise structure of the cover.

\begin{theorem}\label{thm:cover}
  Let $\alpha\in (0,1)\setminus \Q$ and let $\X_\alpha$ be the associated one-sided Sturmian subshift.
  The canonical factor map $\pi_\alpha\colon \tilde{\X}_\alpha \to \X_\alpha$ is
  \begin{enumerate}[label=(\arabic*)]
    \item \label{i:cover3-to-1} three-to-one on the forward orbit $\orb_\alpha^+(\omega_\alpha) \coloneqq \{\sigma_\alpha^k(\omega_\alpha) : k\geq 0 \}$,
    \item \label{i:cover2-to-1} two-to-one on the backward orbit $\orb_\alpha^-(\omega_\alpha) \coloneqq \orb_\alpha(\omega_\alpha)\setminus \orb_\alpha^+(\omega_\alpha)$, and
    \item \label{i:cover1-to-1} one-to-one otherwise.
  \end{enumerate}
  Moreover, $\iota_\alpha(\orb_\alpha(\omega_\alpha)) \subset \tilde{\X}_\alpha$ are exactly the isolated points in $\tilde{\X}_\alpha$.
\end{theorem}

\begin{proof}
  Let $\omega = \omega_\alpha$ be the unique branch point.

  \labelcref{i:cover3-to-1} and \labelcref{i:cover2-to-1}:
  We know from \cref{lem:Sturmian-cover-isolated} that each $x\in \orb_\alpha(\omega)$ defines an isolated point $\iota_\alpha(x)$ in the cover and all isolated points arise like this.
  Every point in the forward orbit $x\in \orb_\alpha^+(\omega)$ has two distinct pasts, and an easy adaptation of the procedure in \cref{construction}
  defines two distinct and nonisolated elements in $\pi_\alpha^{-1}(x)$.
  Similarly, every point in the backward orbit $x\in \orb_\alpha^-(\omega)$ has a unique past, so the procedure of \cref{construction} defines a nonisolated element in $\pi_\alpha^{-1}(x)$.

  It remains to verify that there are no more elements in the fibres.
  We illustrate this for the case of the branch point $\omega = \omega_\alpha$ but the argument applies to every point in the orbit $\orb_\alpha(\omega_\alpha)$.
  If $\tilde{x}\in \pi_\alpha^{-1}(\omega)$ is not one of the elements we constructed above, then each representative ${}_k \tilde{x}_l$
  is the form ${}_k \tilde{x}_l = \omega_{[0,k)} x_{(k,l)}$ for some $x_{(k,l)}\in \X_\alpha$ satisfying $|P_l(x_{(k,l)})| = 2$.
  We show that ${}_k \tilde{x}_l = \omega$, and this will imply that $\tilde{x} = \iota_\alpha(\omega)$.

  Fix integers $0\leq k\leq l$.
  The condition $|P_l(x_{(k,l)})| = 2$ implies that 
  \begin{equation} 
    x_{(k,l)} = \sigma_\alpha^{n_{(k,l)}}(\omega),
  \end{equation}
  for some $n_{(k,l)}\in \N$.
  We will show that $n_{(k,l)} = k$ from which it follows that ${}_k\tilde{x}_l = \omega$.

  Consider integers $0\leq k' \leq l'$ with $(k,l)\preceq (k',l')$.
  The representative is of the form ${}_{k'}\tilde{x}_{l'} = \omega_{[0,k')} x_{(k',l')}$, and we have ${}_{k}\tilde{x}_{l} \widesim{k,l} {}_{k'}\tilde{x}_{l'}$, so
  \begin{equation}
    P_l(\sigma_\alpha^{n_{(k,l)}}(\omega)) = P_l(\omega_{[k,k')} x_{(k',l')}).
  \end{equation}
  Since $\sigma_\alpha^{n_{(k,l)}}(\omega)$ is isolated in $l$-past equivalence, we have $\sigma_\alpha^{n_{(k,l)}}(\omega) = \omega_{[k,k')} x_{(k',l')}$.
  In particular,
  \begin{equation} 
    \sigma_\alpha^{n_{(k,l)}}(\omega) \in Z_\alpha(\omega_{[k, k')}).
  \end{equation}
  As this reasoning applies whenever $(k,l) \preceq (k',l')$ we infer that $\sigma_\alpha^{n_{(k,l)}}(\omega) = \sigma_\alpha^k(\omega)$,
  and since $\omega$ is aperiodic we conclude $n_{(k,l)} = k$ as wanted.

  \labelcref{i:cover1-to-1}:
  Let $x\notin \orb_\alpha(\omega)$ and take $\tilde{x}\in \pi_\alpha^{-1}(x)$ so that ${}_k \tilde{x}_l = x_{[0, k)} x_{(k,l)}$, for some $x_{(k,l)}\in \X_\alpha$.
  We aim to show that ${}_k\tilde{x}_l \widesim{k,l} x$ from which it follows that $\tilde{x} = \iota_\alpha(x)$.
  It only remains to show that 
  \begin{equation}\label{eq:P_l(x_{(k,l)})}
    P_l(x_{(k,l)}) = P_l(\sigma_\alpha^k(x)).
  \end{equation}
  Since $x$ has unique past, we know that $P_l(\sigma_\alpha^k(x)) = \{ \mu x_{[0,k)}\}$, where $\mu\in \LL_\alpha$ is the unique prefix of $x$ with $|\mu| = l - k$.
  
  By \cref{lem:unique-past}, $\tilde{x}$ has a unique past so, in particular, there is a unique $\tilde{x}'\in \tilde{\X}_\alpha$ satisfying $\tilde{\sigma}_\alpha^{l - k}(\tilde{x}') = \tilde{x}$.
  This means that
  \begin{equation} 
    x_{[0,k)} x_{(k,l)} = {}_k\tilde{x}_l \widesim{k,l} \sigma_\alpha^{l - k}({}_{k + (l - k)}\tilde{x}'_l) = \sigma_\alpha^{l - k}({}_l\tilde{x}'_l),
  \end{equation}
  and since $\pi_\alpha(\tilde{x}') = \mu x$, we see that $\mu x_{[0, k)}\in P_l(x_{(k,l)})$.
  Since $P_l(x_{(k,l)})$ is a singleton this proves \labelcref{eq:P_l(x_{(k,l)})}.
\end{proof}

\begin{corollary}
  Let $\tilde{\Lambda}_\alpha \coloneqq \tilde{\X}_\alpha\setminus \iota_\alpha(\orb_\alpha(\omega_\alpha))$ be the collection of nonisolated elements in the cover.
  The restricted dynamical system $(\tilde{\Lambda}_\alpha, \tilde{\sigma}_\alpha)$ is conjugate to the two-sided Sturmian system $(\Lambda_\alpha, \sigma_\alpha)$.
  Consequently, there is an injective sliding block code $i\colon \Lambda_\alpha \to \tilde{\X}_\alpha$ such that $\rho_\alpha = \pi_\alpha\circ i$,
  where $\rho_\alpha$ is the canonical truncation.
 \end{corollary}

\begin{proof}
  We know from~\cref{thm:cover} that $\iota_\alpha(\orb_\alpha(\omega)) = \orb_\alpha(\iota_\alpha(\omega))$ are exactly the isolated points of the cover.
  They form an open and invariant subset, so $\tilde{\Lambda}_\alpha$ is a closed invariant subset homeomorphic to the Cantor set.
  We will first show that $\tilde{\sigma}_\alpha$ restricted to $\tilde{\Lambda}_\alpha$ is a homeomorphism, and then construct a conjugacy $\phi\colon \Lambda_\alpha \to \tilde{\Lambda}_\alpha$.

  We show that $\tilde{\sigma}_\alpha|_{\tilde{\Lambda}_\alpha}$ is injective by showing that every element of $\tilde{\Lambda}_\alpha$ has a unique past.
  Every $x\notin \orb_\alpha^+(\omega)$ has a unique past, so any nonisolated $\tilde{x}\in \pi_\alpha^{-1}(x)$ has a unique past by~\cref{lem:unique-past}.
  It remains to verify that every nonisolated $\tilde{x}\in \pi_\alpha^{-1}(\sigma_\alpha^k(\omega))$ also has a unique past for all $k\in \N$.

  Fix $k\in \N$ and a nonisolated element $\tilde{x}\in \pi_\alpha^{-1}(\sigma_\alpha^k(\omega))$.
  Choose an integer $i > k$ and take $\tilde{z}\in \tilde{\sigma}_\alpha^{-i}(\tilde{x})$ corresponding to the \texttt{0}-past
  (the argument for the \texttt{1}-past is analogous).
  Then $\tilde{z}$ is not isolated, and
  \begin{equation*} 
    \pi_\alpha(\tilde{z}) = \mu \texttt{0} \omega
  \end{equation*}
  where $\mu \texttt{0}\in \LL_\alpha$ is the unique prefix of $\omega$ with $|\mu \texttt{0}| = i - k$.
  But there is a unique element satisfying these properties, so $\tilde{\sigma}_\alpha^{-i}(\tilde{x})$ is a singleton.
  As this applies to all $i > k$, we conclude that $\tilde{x}$ has a unique past.

  From the description of the cover $\tilde{\X}_\alpha$ in \cref{thm:cover}, there is an obvious conjugacy $\phi\colon \Lambda_\alpha\to \tilde{\Lambda}_\alpha$.
  Specifically, 
  \begin{itemize}
    \item if $\textrm{x}\in \Lambda_\alpha$ and $\textrm{x}_{[0,\infty)}\in \X_\alpha$ has a unique past (which is $\textrm{x}$ itself),
      then there is a unique nonisolated element $\tilde{x}\in \pi_\alpha^{-1}(\textrm{x}_{[0,\infty)})$, and $\phi(\textrm{x}) = \tilde{x}$;
    \item if $\textrm{x}\in \Lambda_\alpha$ and $\textrm{x}_{[0,\infty)}\in \X_\alpha$ does not have a unique past 
      (i.e. $\textrm{x}_{[0,\infty)} = \sigma_\alpha^k(\omega)$, for some $k\in \N$, so $\textrm{x}_{[-k,\infty)} = \omega$),
      and if $\textrm{x}_{-(k+1)} = \texttt{0}$, then there is a unique nonisolated element $\tilde{x}\in \pi_\alpha^{-1}(\sigma_\alpha^k(\omega))$
      corresponding to the \texttt{0}-past, and $\phi(\textrm{x}) = \tilde{x}$
      (if $\textrm{x}_{-(k+1)} = \texttt{1}$, then there is a unique nonisolated element $\tilde{x}'\in \pi_\alpha^{-1}(\sigma_\alpha^k(\omega))$
      corresponding to the \texttt{1}-past, and $\phi(\textrm{x}) = \tilde{x}'$).
  \end{itemize}
  Finally, the conjugacy $\phi\colon \Lambda_\alpha \to \tilde{\Lambda}_\alpha$ composed with the inclusion $\tilde{\Lambda}_\alpha \to \tilde{\X}_\alpha$ 
  is an injective sliding block code $i\colon \Lambda_\alpha \to \tilde{\X}_\alpha$.
  From the above description we see that $\rho_\alpha = \pi_\alpha\circ i$, where $\rho_\alpha\colon \Lambda_\alpha \to \X_\alpha$ is the canonical truncation.
\end{proof}

Sturmian subshifts are canonical examples of Cantor minimal systems.
However, we see that the cover associated to the (one-sided) Sturmian subshift cannot be minimal since it contains an invariant countable discrete subset.
Moreover, the shift operation on the cover is not expansive.
If it were, the cover would be conjugate to a shift of finite type and the Sturmian subshift would be sofic.

\section{Groupoids and \texorpdfstring{$C^*$}{C*}-algebras of Sturmian systems} \label{sec:algebras}

In this section we associate a $C^*$-algebra $\OO_\alpha$ to the one-sided Sturmian system $(\X_\alpha, \sigma_\alpha)$ via a groupoid construction as in~\cite{Brix-Carlsen}.
We study their relation to the underlying dynamical systems as well as their intrinsic structure.
Finally, we show that the dynamic asymptotic dimension of the groupoid is one, and that the nuclear dimension of $\OO_\alpha$ is one. 

\begin{remark}
  The two-sided system $(\Lambda_\alpha, \sigma_\alpha)$ admits a transformation groupoid $\Lambda_\alpha\rtimes_{\sigma_\alpha} \Z$,
  and its groupoid $C^*$-algebra $C(\Lambda_\alpha)\rtimes \Z$ is well understood.
  This crossed product is a unital and simple $A\T$-algebra with real rank zero; it has stable rank one and a unique trace, 
  cf.~\cite[p. 60]{GPS1995} (see also~\cite{Putnam1990}) and~\cite[Proposition 4.2]{Putnam-Schmidt-Skau}.
\end{remark}

We shall now consider the $C^*$-algebra associated to the one-sided Sturmian system $(\X_\alpha, \sigma_\alpha)$ as in~\cite{Brix-Carlsen}.
The cover $(\tilde{\X}_\alpha, \tilde{\sigma}_\alpha)$ from~\cref{sec:cover} admits a natural topological groupoid construction (sometimes known as the Deaconu--Renault groupoid)
\begin{equation} 
  \G_\alpha \coloneqq \{ (\tilde{x}, p, \tilde{y})\in \tilde{\X}_\alpha\times \Z\times \tilde{\X}_\alpha \mid
  \exists k,l\in \N: \tilde{\sigma}_\alpha^k(\tilde{x}) = \tilde{\sigma}_\alpha^l(\tilde{y}),~p = k - l  \}.
\end{equation}
The product of $(\tilde{x},p, \tilde{y})$ and $(\tilde{y}', q, \tilde{z})$ is defined exactly if $\tilde{y} = \tilde{y}'$
in which case 
\[
  (\tilde{x},p, \tilde{y}) (\tilde{y}, q, \tilde{z}) = (\tilde{x}, p+q, \tilde{z}),
\]
and inversion is given by $(\tilde{x},p, \tilde{y})^{-1} = (\tilde{y}, -p, \tilde{x})$.
The unit space is canonically identified with $\tilde{\X}_\alpha$, and the source and range maps $s,r\colon \G_\alpha \to \tilde{\X}_\alpha$ are given as
$s(\tilde{x},p, \tilde{y}) = \tilde{y}$ and $r(\tilde{x},p, \tilde{y}) = \tilde{x}$, respectively.

A basis for a topology on $\G_\alpha$ is given by sets of the form
\[
  Z(\tilde{U}, m,n, \tilde{V}) \coloneqq \{(\tilde{x}, m - n, \tilde{y})\in \G_\alpha : \tilde{x}\in \tilde{U},~\tilde{y}\in \tilde{V},
  ~\tilde{\sigma}_\alpha^m(\tilde{x}) = \tilde{\sigma}_\alpha^n(\tilde{y})\},
\]
where $\tilde{U}, \tilde{V}\subset \tilde{X}_\alpha$ are open.
Equivalently, we can ask that $\tilde{\sigma}_\alpha^m|_{\tilde{U}}$ and $\tilde{\sigma}_\alpha^n|_{\tilde{V}}$ are homeomorphisms and
$\tilde{\sigma}_\alpha^m(\tilde{U}) = \tilde{\sigma}_\alpha^n(\tilde{V})$.
This is a second-countable, locally compact Hausdorff and \'etale groupoid~\cite[Lemma 3.1]{Sims-Williams},
and it is principal because $\X_\alpha$ (and hence $\tilde{\X}_\alpha$) contains no periodic points.
By~\cite[Lemma 3.5]{Sims-Williams} $\G_\alpha$ is amenable, so the reduced and full groupoid $C^*$-algebras coincide.

We define the $C^*$-algebra $\OO_\alpha$ of $(\X_\alpha, \sigma_\alpha)$ to be the groupoid $C^*$-algebra $C^*(\G_\alpha)$.
It carries a canonical \emph{diagonal} subalgebra $\D_\alpha \coloneqq C(\tilde{\X}_\alpha) \subset \OO_\alpha$
which is a Cartan subalgebra.
Moreover, the factor map $\pi_\alpha\colon \tilde{\X}_\alpha \to \X_\alpha$ induces an inclusion $C(\X_\alpha) \subset \D_\alpha$ in $\OO_\alpha$.
The $C^*$-algebra $\OO_\alpha$ also admits a canonical gauge action $\gamma^\alpha\colon \T \curvearrowright \OO_\alpha$
which is induced from the canonical cocycle $c\colon \G_\alpha \to \Z$ given by $c(\tilde{x}, p, \tilde{y}) = p$ (though we shall not need it here).

From our study of the cover in~\cref{sec:cover}, we immediately get the following result.
\begin{lemma} 
  Let $\alpha\in (0,1)\setminus \Q$.
  The groupoid $\G_\alpha$ of $(\X_\alpha, \sigma_\alpha)$ is ample, principal, and not minimal,
  and the unit space decomposes into invariant subsets $\tilde{\X}_\alpha = \orb_\alpha(\tilde{\omega}_\alpha) \cup \tilde{\Lambda}_\alpha$ such that
  \begin{enumerate}
    \item $\orb_\alpha(\tilde{\omega}_\alpha) \subset \tilde{\X}_\alpha$ is an open and invariant subset,
      and $\G_\alpha$ restricted to this open subset is (isomorphic to) the full equivalence relation on a countably infinite set; 
    \item $\tilde{\Lambda}_\alpha \subset \tilde{\X}_\alpha$ is a closed and invariant subset,
      and $\G_\alpha$ restricted to this closed subset is (isomorphic to) the crossed product $\Lambda_\alpha\rtimes \Z$.
  \end{enumerate}
\end{lemma}

As a corollary, we can now recover an unpublished result of Carlsen~\cite[Theorem 8.18]{Carlsen2006} for Sturmian subshifts
in a form where we understand the maps already at the levels of the underlying dynamical systems and with a different proof.

\begin{corollary} \label{cor:extension}
  For $\alpha\in (0,1)\setminus \Q$ there is a commutative diagram
  \begin{equation}
    \begin{tikzcd}
      0 \arrow[r] & c_0 \arrow[r]\arrow[d] & \D_\alpha \arrow[r]\arrow[d] & C^*(\Lambda_\alpha) \arrow[r]\arrow[d] & 0 \\
      0 \arrow[r] & \K \arrow[r] & \OO_\alpha \arrow[r] & C(\Lambda_\alpha)\rtimes \Z \arrow[r] & 0
    \end{tikzcd}
  \end{equation}
  where the rows are short exact and the vertical arrows are canonical inclusions.
\end{corollary}

\begin{proof}
  The groupoid $\G_\alpha$ restricted to the open invariant subset $\orb_\alpha(\tilde{\omega}_\alpha)$ is the full equivalence relation on a countably infinite set,
  so its $C^*$-algebra is isomorphic to the compact operators, cf.~e.g~\cite[Example 9.3.7]{Sims-notes}.
  Furthermore, the groupoid restricted to $\tilde{\Lambda}_\alpha$ on which the dynamics is reversible, is isomorphic to the transformation groupoid $\Lambda_\alpha\rtimes \Z$
  whose $C^*$-algebra is the crossed product $C(\Lambda_\alpha)\rtimes \Z$.
  The fact that the horizontal sequences are exact follows, e.g. from~\cite[Proposition 10.3.2]{Sims-notes}.
\end{proof}

The compact operators is the unique ideal of $\OO_\alpha$.

\begin{remark}
  The $K$-theory of $\OO_\alpha$ was computed in~\cite[Example 5.3]{Carlsen-Eilers2006} and~\cite[Corollary 4.1]{Eilers-Restorff-Ruiz} to be
  \begin{equation} 
    K_0(\OO_\alpha) \cong \Z + \alpha \Z, \quad \textrm{and} \quad K_1(\OO_\alpha) = 0.
  \end{equation}
  The first isomorphism is of ordered groups when $\Z + \alpha \Z$ inherits the order of the real line.
\end{remark}

We can now relate existence of $^*$-isomorphisms or Morita equivalence between these $C^*$-algebras to the dynamics of the underlying subshifts
by using the results of~\cref{sec:dynamics}.
Note that similarity with~\cite[Theorems 2 and 4]{Rieffel1981}

\begin{proposition} 
  Let $\alpha, \beta\in (0,1)\setminus \Q$.
  The two-sided Sturmian systems $(\Lambda_\alpha, \sigma_\alpha)$ and $(\Lambda_\beta, \sigma_\beta)$ are two-sided conjugate 
  (equivalently, the one-sided Sturmian subshifts $(\X_\alpha, \sigma_\alpha)$ and $(\X_\beta, \sigma_\beta)$ are one-sided conjugate) if and only if 
  the $C^*$-algebras $\OO_\alpha$ and $\OO_\beta$ are $^*$-isomorphic if and only if
  $\alpha = \beta$ or $\alpha = 1 - \beta$.
\end{proposition}

\begin{proof}
  The two-sided subshifts $(\Lambda_\alpha, \sigma_\alpha)$ and $(\Lambda_\beta, \sigma_\beta)$ are conjugate if and only if 
  the one-sided subshifts $(\X_\alpha, \sigma_\alpha)$ and $(\X_\beta, \sigma_\beta)$ are one-sided conjugate,
  and this happens exactly when $\alpha = \beta$ or $\alpha = 1 - \beta$.
  By~\cite[Theorem 4.4]{Brix-Carlsen} there is a $^*$-isomorphism from $\OO_\alpha$ and $\OO_\beta$ which is diagonal-preserving and intertwines the gauge actions. 
  Conversely, if $\OO_\alpha$ and $\OO_\beta$ are $^*$-isomorphic, then their $K$-theory agrees;
  in fact, $\Z + \alpha \Z$ and $\Z + \beta \Z$ are order isomorphic in a way which preserves the distinguished unit. 
  Therefore, $\alpha = \beta$ or $\alpha = 1 - \beta$.
\end{proof}

\begin{proposition} 
  Let $\alpha, \beta\in (0,1)\setminus \Q$.
  The two-sided Sturmian systems $(\Lambda_\alpha, \sigma_\alpha)$ and $(\Lambda_\beta, \sigma_\beta)$ are flow equivalent if and only if 
  $\OO_\alpha$ and $\OO_\beta$ are Morita equivalent if and only if
  $\alpha \sim \beta$ (equivalence of irrationals).
\end{proposition}

\begin{proof}
  The systems $(\Lambda_\alpha, \sigma_\alpha)$ and $(\Lambda_\beta, \sigma_\beta)$ are flow equivalent exactly when $\alpha$ and $\beta$ are equivalent irrationals. 
  Flow equivalence implies the existence of a $^*$-isomorphism from $\OO_\alpha \otimes \K$ to $\OO_\beta \otimes \K$ which is diagonal-preserving;
  in particular, $\OO_\alpha$ and $\OO_\beta$ are Morita equivalent.
  Conversely, if $\OO_\alpha$ and $\OO_\beta$ are Morita equivalent, then their ordered $K$-theory agrees, so the irrationals $\alpha$ and $\beta$ are equivalent.
\end{proof}

These results tell us that the class $\{ \OO_\alpha \}_\alpha$ of $C^*$-algebras associated to Sturmian systems contains uncountably many non-$^*$-isomorphic (even non-Morita equivalent) $C^*$-algebras,
and the ordered $K$-theory classifies them up to Morita equivalence, cf.~\cite[Corollary 4.1]{Eilers-Restorff-Ruiz}.

\begin{remark}
  The diagonal subalgebra $\D_\alpha$ inside $\OO_\alpha$ is a Cartan subalgebra in the sense of \cite{Renault2008} and it is unique in the following sense:
  if $\OO_\alpha$ and $\OO_\beta$ are $^*$-isomorphic, then the systems $(\X_\alpha, \sigma_\alpha)$ and $(\X_\beta, \sigma_\beta)$ are conjugate (by the above result)
  so there is a diagonal-preserving isomorphism from $\OO_\alpha$ to $\OO_\beta$, i.e. a $^*$-isomorphism which maps the diagonal onto the diagonal, cf.~\cite[Theorem 4.1]{Brix-Carlsen}.
  This is reminiscent of AF-algebras and not something to expect among Cantor minimal system, cf.~\cite[pp. 65--66]{GPS1995}.

  Moreover, we may deduce a \emph{stable $^*$-isomorphism implies $^*$-isomorphism} result.
  If a $^*$-isomorphism from $\OO_\alpha\otimes \K$ to $\OO_\beta\otimes \K$ which maps $C(\X_\alpha)\otimes c_0$ onto $C(\X_\beta)\otimes c_0$
  can be chosen to also intertwine the gauge actions $\gamma^\alpha\otimes \id$ and $\gamma^\beta\otimes \id$,
  then $(\Lambda_\alpha, \sigma_\alpha)$ and $(\Lambda_\beta, \sigma_\beta)$ are conjugate (cf.~\cite[Corollary 7.6]{Brix-Carlsen}) so $\OO_\alpha$ and $\OO_\beta$ are $^*$-isomorphic.
\end{remark}

\begin{lemma} 
  The $C^*$-algebra $\OO_\alpha$ is infinite (though not properly infinite), has real rank zero, and stable rank two for any $\alpha\in (0,1)\setminus \Q$.
\end{lemma}

\begin{proof}
  The fact that $\OO_\alpha$ has real rank zero follows from~\cite[Theorem 3.14]{Brown-Pedersen}.
  Since both the compacts and the crossed product has stable rank one, it follows from~\cite[Corollary 4.12]{Rieffel1983} and \cref{cor:extension} that the stable rank of $\OO_\alpha$ is one or two.
  The existence of an infinite projection in $\OO_\alpha$ shows that it is an infinite $C^*$-algebra and that the stable rank must be two 
  (stable rank one would imply that $\OO_\alpha$ were stably finite).

  Let $B \subset \G_\alpha$ be the union of $\tilde{\Lambda}_\alpha$, $\{ \iota_\alpha(\sigma_\alpha^{-k}(\texttt{0} \omega_\alpha) : k\in \N)\}$,
  $\{ \iota_\alpha(\sigma_\alpha^k(\omega_\alpha)) : k \in \N \}$, and
  \begin{equation} 
    \{ (\iota_\alpha(\nu 1 \omega_\alpha), 1, \iota_\alpha(\sigma_\alpha(\nu 1 \omega_\alpha))) : |\nu| \geq 1 \}.
  \end{equation}
  Then $B$ is an open bisection with $s(B) = \tilde{\X}_\alpha$ and $r(B) = \tilde{\X}_\alpha \setminus \{ \texttt{1} \tilde{\omega}_\alpha\}$.
  Therefore, the indicator function $\chi_B \in \OO_\alpha$ is a proper isometry, so $\OO_\alpha$ is infinite.
  Since $K_0(\OO_\alpha)$ is an ordered group, $\OO_\alpha$ is not properly infinite.
\end{proof}

Finally we turn our attention to the nuclear dimension for $C^*$-algebras.
This was introduced in~\cite{Winter-Zacharias} as a noncommutative analogue of topological covering dimension,
and it has proven to be immensely valuable to the classification programme.
Since $(\Lambda_\alpha, \sigma_\alpha)$ is a Cantor minimal system it is well known that its crossed product $C^*$-algebra has nuclear dimension one (since it is not AF).
By general results on extensions of $C^*$-algebras (\cite[Proposition 2.9]{Winter-Zacharias}),
it follows from~\cref{cor:extension} that the nuclear dimension of $\OO_\alpha$ is at most two.
We will show that the nuclear dimension is one by first determining the dynamic asymptotic dimension of the groupoid $\G_\alpha$ (\cite[Section 5]{Guentner-Willett-Yu}).

\begin{definition}
  An \'etale groupoid $\G$ has \emph{dynamic asymptotic dimension} at most $d\in \N$ if for any open and relatively compact subset $K\subset \G$
  there exist open subsets $U_0,\dots,U_d$ in $\G^{(0)}$ such that $s(K) \cup r(K) \subset \bigcup_i U_i$ and such that the groupoid generated by the set
  \begin{equation} 
    K_{U_i} \coloneqq \{ g\in G : s(g), r(g) \in U_i \}
  \end{equation}
  is relatively compact in $\G$ for each $i = 0,\dots,d$.
\end{definition}

The crossed product groupoid of $(\Lambda_\alpha, \sigma_\alpha)$ has dynamic asymptotic dimension one by~\cite[Theorem 1.4(i)]{Guentner-Willett-Yu}.

In the proof below we shall use the notation
\begin{equation} 
  Z\big(\sigma_\alpha^j(\mu) : j = 0,\dots,n\big) \coloneqq \bigcup_{j = 0}^n Z\big(\sigma_\alpha^j(\mu)\big) 
\end{equation}
for a word $\mu\in \LL_\alpha$ and an integer $n \leq |\mu|$.
Moreover, if $\mu$ and $\nu$ are two words, then we write $\mu \preceq \nu$ to say that $\mu$ appears somewhere in $\nu$.

\begin{theorem} \label{thm:dad}
  The dynamic asymptotic dimension of $\G_\alpha$ is one for any $\alpha\in (0,1)\setminus \Q$.
\end{theorem}

\begin{proof}
Let $K \subset \G_\alpha$ be an open and relatively compact subset. 
In particular, $\bar{K} \subset \bigcup_{i} Z(\tilde{\X}_\alpha, k_i, l_i, \tilde{\X}_\alpha)$ for some finite union.
Note that 
\begin{equation} 
  Z(\tilde{\X}_\alpha, k_i, l_i, \tilde{\X}_\alpha) = Z(\tilde{\X}_\alpha, 0, k_i - l_i, \tilde{\X}_\alpha) \cup Z(\orb(\tilde{\omega}), k_i, l_i, \orb(\tilde{\omega}))
\end{equation}
if $k_i - l_i \leq 0$ (of course, a similar thing happens if $k_i - l_i \geq 0$),
and that the groupoid generated by the right most set is finite.
Since the groupoid generated by a subset $K$ only depends on $K\cup K^{-1}\cup s(K)\cup r(K)$,
we may therefore assume that 
\begin{equation} 
  K = \bigcup_{l\in F} Z(\tilde{\X}_\alpha, 0, l, \tilde{\X}_\alpha),
\end{equation}
for some finite subset $F \subset \N$.

Let $\bar{l} = \max\{l : l\in F\}$ and choose distinct words $\mu', \nu'$ of length $\bar{l}$.
Extend both words to the left (in any way) to obtain words $\mu, \nu$ of length $2 \bar{l}$.
Now let 
\begin{equation} 
  \tilde{U} \coloneqq \pi_\alpha^{-1}\big(Z(\sigma_\alpha^j(\mu) : j = 0,\dots,\bar{l}-1)\big) \quad \textrm{and}\quad \tilde{V} \coloneqq \tilde{\X} \setminus \tilde{U},
\end{equation}
and observe that $\pi_\alpha^{-1}\big(Z(\sigma_\alpha^j(\nu) : j = 0,\dots,\bar{l}-1)\big) \subset \tilde{V}$. 
Then $\tilde{U}$ and $\tilde{V}$ are disjoint open sets which cover $\tilde{\X}$.
Let 
\begin{equation} 
  K_{\tilde{U}} \coloneqq \{ g\in K : s(g),r(g)\in \tilde{U} \} \quad \textrm{and} \quad   K_{\tilde{V}} \coloneqq \{ g\in K : s(g),r(g)\in \tilde{V} \}.
\end{equation}
We shall verify that the groupoids generated by these two sets in $\G_\alpha$ are relatively compact, and this will prove the claim.

If there is a bound on the number of elements in $\{g\in K_{\tilde{V}} : c(g) \geq 0\}$ that can be concatenated, say $\beta\in \N$,
then the groupoid generated by $K_{\tilde{V}}$ is contained in
\begin{equation}  \label{eq:compact-set}
  \bigcup_{k_i,l_i \leq \bar{k} \beta} Z(\tilde{\X}_\alpha, k_i, l_i, \tilde{X}_\alpha)
\end{equation}
which is compact (being a finite union of compact sets).

Suppose that $g\in K$ with $\tilde{x} \coloneqq s(g), r(g)\in \tilde{V}$, and let $x \coloneqq \pi_\alpha(\tilde{x}) \in \X_\alpha$.
Let $\beta(\mu)\in \N$ be such that $\mu \preceq x_{[0, \beta(\mu)]}$, and note that this number can be chosen to be independent of $x$.
After applying at most $\beta(\mu)$ shifts to $\tilde{x}$,
there is an element $\tilde{x}'$ such that 
\[
  x' \coloneqq \pi_\alpha(\tilde{x}') \in Z\big(\sigma_\alpha^j(\mu) : j = 0,\dots,\bar{l} - 1\big).
\]
So $\tilde{x}' \in \tilde{U}$.
This means that there can be at most $\beta(\mu)$ concatenations of elements in $\{g\in K_{\tilde{V}} : c(g) \geq 0\}$.
A similar argument applies to $\{ g\in K_{\tilde{V}} : c(g)\leq 0\}$.
Therefore, the groupoid generated by $K_{\tilde{V}}$ is contained in a compact set of the form~\labelcref{eq:compact-set} so it is relatively compact.

A similar argument applies to $g\in K$ with $\tilde{x} \coloneqq s(g), r(g) \in \tilde{U}$.
Let $x \coloneqq \pi_\alpha(\tilde{x}) \in Z(\sigma_\alpha^j(\mu) : j = 0,\dots,\bar{k}-1)$.
If $\beta(\nu)\in \N$ is a number such that $\nu \preceq x_{[0, \beta(\nu)]}$ (independent of $x$),
then after at most $\beta(\nu)$ shifts of $\tilde{x}$, 
there is an element $\tilde{x}'$ with $x' \coloneqq \pi_\alpha(\tilde{x}') \in Z(\sigma_\alpha^j(\nu) : j = 0,\dots,\bar{l} - 1)$.
As above, it follows that the groupoid generated by $K_{\tilde{U}}$ is contained in a set of the form~\labelcref{eq:compact-set} and is hence relatively compact.
\end{proof}

\begin{corollary} 
  The nuclear dimension of $\OO_\alpha$ is one for every $\alpha\in (0,1)\setminus \Q$.
\end{corollary}

\begin{proof}
  The unit space $\tilde{\X}_\alpha$ of $\G_\alpha$ is zero-dimensional, so it follows from ~\cref{thm:dad} and~\cite[Theorem 8.6]{Guentner-Willett-Yu} that 
  the nuclear dimension of $\OO_\alpha$ is one.
\end{proof}


\end{document}